\documentclass[12pt]{amsart}
\usepackage{amscd,amssymb,hyperref}
\usepackage{graphicx}
\DeclareGraphicsExtensions{.pdf}

\newtheorem{thm}{Theorem}[section]
\newtheorem{lem}[thm]{Lemma}

\newtheorem{prop}[thm]{Proposition}

\newtheorem{question}[thm]{Question}
\def\square{\vbox{
      \hrule height 0.4pt
      \hbox{\vrule width 0.4pt height 5.5pt \kern 5.5pt \vrule width 0.4pt}
      \hrule height 0.4pt}}

\def\ch\mathrm{c h}

\long\def\symbolfootnote[#1]#2{\begingroup%
\def\thefootnote{\fnsymbol{footnote}}\footnote[#1]{#2}\endgroup}

\newcommand{\Z}{\mathbb{Z}}

\numberwithin{equation}{section}

\newcommand{\auths}[1]{\textrm{#1},}
\newcommand{\artTitle}[1]{\textsl{#1},}

\title{On homotopy  braids}

\author{V. G. Bardakov}
\address{Novosibirsk State University, Pirogova str. 1, Novosibirsk, 630090, Russia}
\address{Sobolev Institute of Mathematics, prosp. Koptyuga 4, Novosibirsk 630090, Russia}
\address{Novosibirsk State Agrarian University, Dobrolyubova street, 160, Novosibirsk, 630039, Russia}
\address{Regional Scientific and Educational Mathematical Center of Tomsk State University, 36 Lenin Ave., Tomsk, Russia}
\email{bardakov@math.nsc.ru}

\author{V. V. Vershinin}
\address{D\'epartement des Sciences Math\'ematiques,
                               Universit\'e Montpellier II,
Place Eug\`ene Bataillon,
34095 Montpellier cedex 5, France} \email{vladimir.verchinine@umontpellier.fr}
\address{Sobolev Institute of Mathematics, Novosibirsk 630090,
Russia }
\email{ versh@math.nsc.ru}

\author{J. Wu}
\address{Center for Topology and Geometry based Technology, Hebei Normal University, No.20 Road East. 2nd Ring South, Yuhua District, Shijiazhuang, Hebei, 050024 CHINA}
\address{School of Mathematical Sciences, Hebei Normal University, No.20 Road East. 2nd Ring South, Yuhua District, Shijiazhuang, Hebei, 050024 CHINA} \email{wujie@hebtu.edu.cn}


\begin{document}

\begin{abstract}
The homotopy braid group $\widehat{B}_n$ is the subject of the paper. First, linearity 
of $\widehat{B}_n$ over the integers is proved. Then we prove that the group $\widehat{B}_3$ is torsion free.

\end{abstract}

\maketitle


\section{Introduction}

The homotopy braid groups are one of the interesting variations of classical braid groups.

Two geometric braids with the same endpoints are called {\it homotopic} if one can be 
deformed to
the other by  homotopies of the braid strings  which fix the endpoints,
so that different strings do not intersect.  If two geometric braids are isotopic, they are
evidently homotopic.  E.~ Artin \cite{Art2} posed the question of whether the
notions of isotopy and homotopy of braids are different or the same. 
Namely he wrote:

''Assume that two braids can be deformed into each other by a deformation of the most 
general nature including self intersection of each string but avoiding 
intersection of two different strings. Are they isotopic?"

Deborah Goldsmith \cite{G}
gave an example of a braid which is not trivial in the isotopic sense, but is homotopic 
to the
trivial braid.
At first she expressed this braid and homotopy process by the  pictures. We give these pictures in Figure~1.

\begin{figure}
\center{\includegraphics[width=1.0\linewidth]{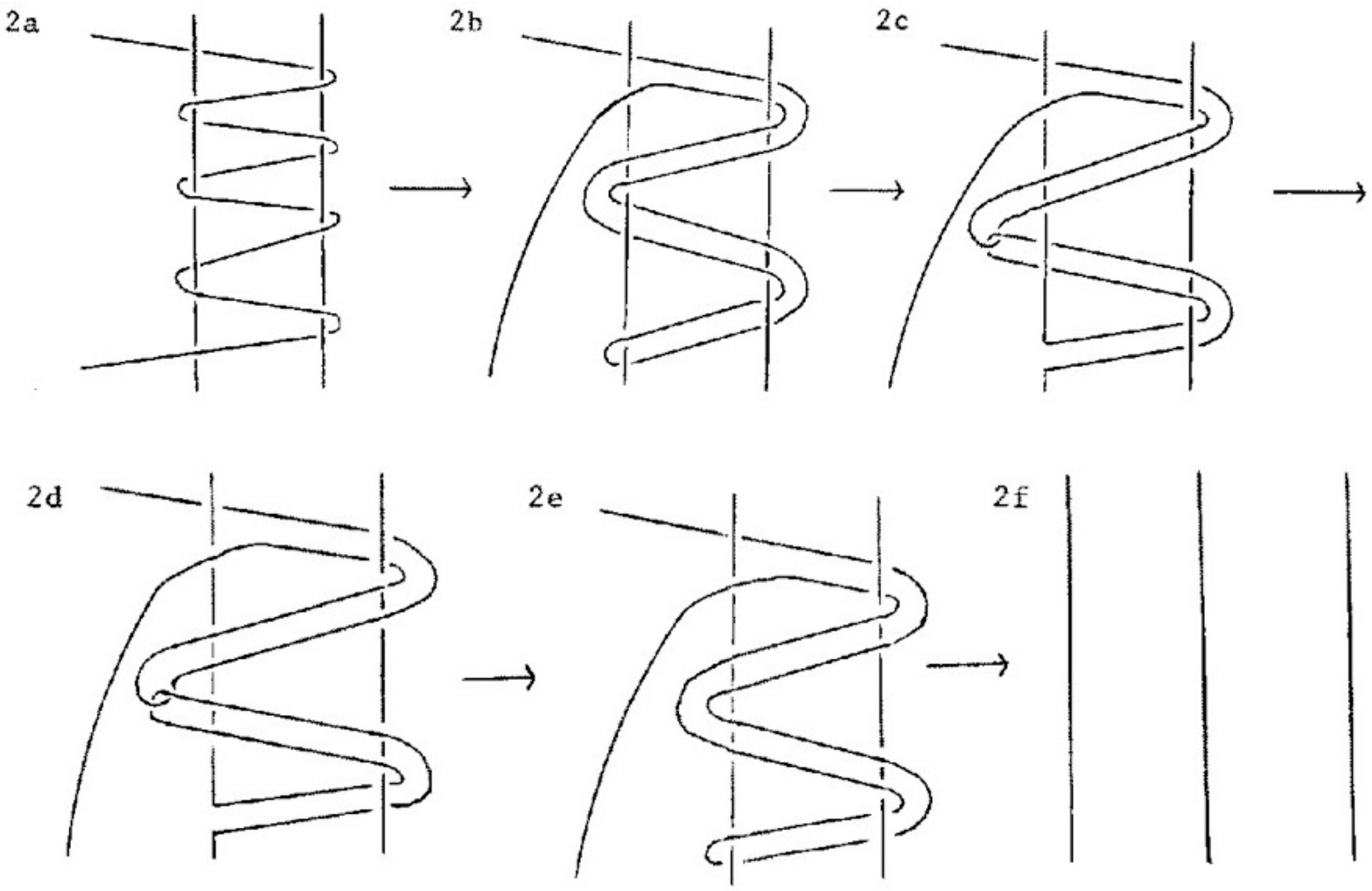}}
\caption{}
\end{figure}
This braid is expressed in the canonical generators of the classical braid group in the 
following form:

$$
\sigma_1\sigma_2^{2}\sigma_1^{2}\sigma_2^{-2}\sigma_1^{-2}\sigma_2^{2}
\sigma_1^{-2}\sigma_2^{-2}\sigma_1.
$$
Thus, the study of homotopy braids seem to bes interesting in itself.

In this work we consider several questions concerning homotopy braids.
The paper is organized as follows.
In section~2 we basically recollect necessary facts about the reduced free groups. In section~3 we prove the linearity of
$\widehat{B}_n$ over $\Z$.
Unfortunately we don't know a concrete representation in $GL_m(\Z)$.
We then prove that the 
Burau representation cannot serve for these purposes.
In section~4 we prove that the group  $\widehat{B}_3$ is torsion free.
Several questions are proposed in the last section.

\section{Reduced free group and homotopy braid group}

For elements $a, b$ of arbitrary group $G$ we will use the following notations
$$
a^b = b^{-1} a b,~~~[a, b] = a^{-1} b^{-1} a b.
$$

Recall some facts from the theory of braids (see, for example, \cite{Art1}, \cite{Art2}, \cite{Bi}, \cite{Mar}). 
The braid group $B_n$, $n\geq 2$, on $n$ strings can be defined as
a group generated by $\sigma_1,\sigma_2,\ldots,\sigma_{n-1}$ with the defining relations
\begin{center}
$\sigma_i \, \sigma_{i+1} \, \sigma_i = \sigma_{i+1} \, \sigma_i \, \sigma_{i+1},~~~ i=1,2,\ldots,n-2, $
\end{center}
\begin{center}
$\sigma_i \, \sigma_j = \sigma_j \, \sigma_i,~~~|i-j|\geq 2. $
\end{center}

There exists a homomorphism of $B_n$ onto the symmetric group $S_n$ on
$n$ letters. This homomorphism  maps
 $\sigma_i$ to the transposition  $(i,i+1)$, $i=1,2,\ldots,n-1$.
The kernel of this homomorphism is called the
{\it pure braid group} and it is denoted by
$P_n$. The group $P_n$ is generated by  the elements $a_{ij}\in B_n$, $1\leq i < j\leq n$.
These
generators can be expressed by the generators of
 $B_n$ as follows
$$
a_{i,i+1}=\sigma_i^2,
$$
$$
a_{ij} = \sigma_{j-1} \, \sigma_{j-2} \ldots \sigma_{i+1} \, \sigma_i^2 \, \sigma_{i+1}^{-1} \ldots
\sigma_{j-2}^{-1} \, \sigma_{j-1}^{-1},~~~i+1< j \leq n.
$$

The subgroup $P_n$ is normal in $B_n$, and the quotient $B_n / P_n$ is the symmetric group $S_n$. 
The generators of $B_n$ act on the generator $a_{ij} \in P_n$ by the rules:
 \begin{center}
$\sigma_k^{-1} a_{ij} \sigma_k =  a_{ij},  ~\mbox{for}~k \not= i-1, i, j-1, j,$ \\
$\sigma_{i}^{-1} a_{i,i+1} \sigma_{i} =  a_{i,i+1}, $  \\
$ \sigma_{i-1}^{-1} a_{ij} \sigma_{i-1} =   a_{i-1,j},$   \\
$ \sigma_{i}^{-1} a_{ij} \sigma_{i} =  a_{i+1,j} [a_{i,i+1}^{-1}, a_{ij}^{-1}],  ~\mbox{for}~j \not= i+1, $\\
$ \sigma_{j-1}^{-1} a_{ij} \sigma_{j-1} =  a_{i,j-1},$   \\
 $\sigma_{j}^{-1} a_{ij} \sigma_{j} =  a_{ij} a_{i,j+1} a_{ij}^{-1}.$
\end{center}

Let us denote by
$$
U_{i} = \langle a_{1i}, a_{2i}, \ldots, a_{i-1,i} \rangle,~~~i = 2, \ldots, n,
$$
the subgroup of $P_n$.
It is known that $U_i$ is the free group of rank $i-1$. The pure braid group  $P_n$ can be given by the relations  (for $\varepsilon = \pm 1$):
 \begin{center}
$a_{ik}^{-\varepsilon} a_{kj}  a_{ik}^{\varepsilon} = (a_{ij} a_{kj})^{\varepsilon} a_{kj} (a_{ij} a_{kj})^{-\varepsilon}, $ \\
$ a_{km}^{-\varepsilon} a_{kj}  a_{km}^{\varepsilon} = (a_{kj} a_{mj})^{\varepsilon} a_{kj} (a_{kj} a_{mj})^{-\varepsilon},  ~\mbox{for}~m < j, $ \\
$ a_{im}^{-\varepsilon} a_{kj}  a_{im}^{\varepsilon} = [a_{ij}^{-\varepsilon}, 
a_{mj}^{-\varepsilon}]^{\varepsilon} a_{kj} [a_{ij}^{-\varepsilon}, a_{mj}^{-\varepsilon}]^{-\varepsilon},  
~\mbox{for}~i < k < m, $ \\
$a_{im}^{-\varepsilon} a_{kj} a_{im}^{\varepsilon} = a_{kj},  ~\mbox{for}~k < i < m < j ~\mbox{or}~  m < k. $
 \end{center}
The group $P_n$ is the semi--direct product of  the normal subgroup
$U_n$ and $P_{n-1}$. Similarly, $P_{n-1}$ is the semi--direct product of the free group
$U_{n-1}$  and $P_{n-2},$ and so on.
Therefore, $P_n$ is decomposable  into the following semi--direct products
$$
P_n=U_n\rtimes (U_{n-1}\rtimes (\ldots \rtimes
(U_3\rtimes U_2))\ldots),~~~U_i\simeq F_{i-1}, ~~~i=2,3,\ldots,n.
$$

Let $F_n =F(x_1,\dots, x_n)$ be the free group on generators $x_1, \dots, x_n$.
We denote
by  $K_n$   the quotient group of $F_n$ by the relations 
$$
[x_i, x_i^g] = 1,   \   i  = 1, \dots, n,  
$$
where $g$ is an arbitrary element of $F_n$.
The group $K_n$ is called the {\it reduced} {\it free group}. 
 So, it is the quotient group of the free group  obtained by adding relations which express that each $x_i$
 commutes with all of its conjugates.
This group can be characterized also the following way.
Let $X_i$ be the normal subgroup of $F_n$   generated by $x_i$ and let $[X_i]$
be the commutator subgroup of $X_i$. Then $N_n = [X_1 ] \dots [X_n]$ is also the normal
subgroup of $F_n$ and 
$K_n$ is  the quotient group $F_n / N_n$. 
  This group  was introduced by
 J.~Milnor \cite{Mi} and studied by Habegger \& Lin \cite{HL}, F.~Cohen~\cite{C2} and  F.~Cohen \& Jie Wu \cite{CW}.

Let $a_{i,j}$ be the standard 
generators of the pure braid group mentioned above.
Recall that the homotopy braid group $\widehat{B}_n$ is the quotient of the braid group $B_n$ by the relations
$$
[a_{ik}, a_{ik}^g] = 1, ~\mbox{where}~g \in \langle a_{1k}, a_{2k}, \ldots, a_{k-1,k} \rangle, 1 \leq i < k \leq n.
$$
Let us denote by $\phi$ the canonical epimorphism from the standard braid group to the homotopy braid group
$$
\phi:B_n\to \widehat{B}_n. 
$$
The quotient of the pure braid group $P_n$ by the same relations gives us the pure homotopy braid group 
$\widehat{P}_n$ and from the standard short exact sequence for $B_n$ we have the following short exact sequence
$$
1 \longrightarrow \widehat{P}_n \longrightarrow \widehat{B}_n \longrightarrow S_n \longrightarrow 1,
$$
where $S_n$ is the symmetric group. The group $\widehat{P}_n$ has the decomposition 
$\widehat{P}_n = \widehat{U}_n  \rtimes \widehat{P}_{n-1}$,
where $\widehat{U}_n$ is the quotient of the free group $U_n = \langle a_{1n}, a_{2n}, \ldots, a_{n-1, n} \rangle$ of rank $n-1$ by the relations
$$
[a_{in}, a_{in}^g] = 1, ~\mbox{where}~g \in U_n, 1 \leq i < k \leq n,
$$
(see \cite{HL}).
Note, that $\widehat{U}_n$ is isomorphic to $K_{n-1}$.
In particular, $\widehat{U}_2$ is isomorphic to the infinite cyclic group and $\widehat{U}_3$ is the quotient of $U_3 = \langle a_{13}, a_{23} \rangle$
by the relations
$$
a_{13} \cdot a_{23}^{-1} a_{13} a_{23} = a_{23}^{-1} a_{13} a_{23}  \cdot  a_{13},
$$
$$
a_{23} \cdot a_{13}^{-1} a_{23} a_{13} = a_{13}^{-1} a_{23} a_{13}  \cdot  a_{23}.
$$

It was proved by Cohen and Wu \cite{CW} that the canonical Artin monomorphism
 $$
\nu_n: B_n  \hookrightarrow \operatorname{Aut} F_n
 $$
induces a homomorphism

 $$
 \hat \nu_n:\widehat B_n \to\operatorname{Aut} K_n.
 $$
It is known that $\hat\nu_n$ is a monomorphism. 
See, for example,  the thesis of
 Liu Minghui \cite{Ming}.

Since $K_n$ is a finitely generated nilpotent group of class $n$ 
\cite[Lemma 1.3]{HL}, \cite{Le}, then it follows from the result of A.I.Mal'cev \cite{Ma}  that
the word problem is decidable  in $K_n$. From the fact that $\widehat{B}_n$ is a finite extension of $\widehat{P}_n$ 
it follows that the word problem is decidable  in
$\widehat{B}_n$. In the next section we prove a stronger result that the group $\widehat{B}_n$ is  linear 
over $\Z$.
 
\section{Linearity problem}

\subsection{Faithful linear representation}

Recall that a group $G$ is called {\it linear} if it has a faithful representation into the general linear group $GL_m(k)$ for some $m$ and a field $k$.
In the works \cite{Big} and \cite{Kra} it was proved that the braid group $B_n$ is linear for every $n \geq 2$. 
So, it is natural to ask whether $\widehat{B}_n$ is linear.

\begin{thm}
The homotopy braid group $\widehat{B}_n$ is linear for all $n \geq 2$. Moreover, for every 
$n \geq 2$ there is a natural $m$ such that there exists a faithful representation $\widehat{B}_n \longrightarrow GL_m(\mathbb{Z})$.
\end{thm}
\begin{proof}
As it was mentioned above, the reduced free group $K_n$, $n \geq 2$ is nilpotent.  Finitely generated nilpotent group is polycyclic and hence
it can be represented by integer matrices \cite{Aus, Sw}.
It was proved in \cite{Mer} that the holomorph of every polycyclic group has a faithful representation into 
$GL_m(\mathbb{Z})$ for some $m$. Hence, holomorph $\operatorname{Hol}(K_n)$ has a faithful representation 
into $GL_m(\mathbb{Z})$ for some $m$. But $\operatorname{Hol}(K_n)$ contains 
$\operatorname{Aut}(K_n)$ as a subgroup and so, $\widehat{B}_n$ is embedded into  $\operatorname{Aut}(K_n)$.
\end{proof}

It is interesting to find a faithful linear representation of $\widehat{B}_n$ explicitly. 
One can try to factor through $\widehat{B}_n$ the known representations of $B_n$, 
for example, Burau representation, Lawrence-Krammer-Bigelow representation or other ones.

\subsection{Burau representation and  the homotopy braid groups}
Let
$$
\rho_B : B_n \longrightarrow GL(W_n)
$$
be the Burau representation of $B_n$, where $W_n$ is a free $\mathbb{Z}[t^{\pm 1}]$-module of rank $n$ with the basis $w_1, w_2, \ldots, w_n$.
Let $n=3$. In this case the automorphisms $\rho_B (\sigma_i)$, $i = 1, 2$, of module $W_3$ act by the rule
$$
\sigma_1 :
\left\{
\begin{array}{l}
w_1 \longmapsto (1-t) w_1 + t w_2, \\
w_2 \longmapsto w_1, \\
w_3 \longmapsto w_3,
\end{array}
\right.
~~~~
\sigma_2 :
\left\{
\begin{array}{l}
w_1 \longmapsto w_1,\\
w_2 \longmapsto (1-t) w_2 + t w_3, \\
w_3 \longmapsto w_2, \\
\end{array}
\right.
$$
where we write for simplicity $\sigma_i$ instead of $\rho_B (\sigma_i)$. Let us find the action of the generators of $P_3$ on the module $W_3$.
Recall, that $P_3 = U_2 \rtimes U_3$, where $U_2$ is the infinite cyclic group with
the generator $a_{12} = \sigma_1^2$, $U_3$ is the free group of rank $2$ with the free generators
$$
a_{13} = \sigma_2 \sigma_1^2 \sigma_2^{-1},~~~a_{23} = \sigma_2^2.
$$
These elements define the following automorphisms of $W_3$
\begin{equation} \label{a12}
a_{12} :
\left\{
\begin{array}{l}
w_1 \longmapsto (1-t+t^2) w_1 + t (1 - t) w_2, \\
w_2 \longmapsto (1 - t) w_1 + t w_2, \\
w_3 \longmapsto w_3,
\end{array}
\right.
\end{equation}
\begin{equation} \label{a13}
a_{13} :
\left\{
\begin{array}{l}
w_1 \longmapsto (1 - t + t^2) w_1 + t (1 - t) w_3,\\
w_2 \longmapsto (1-t)^2 w_1 +  w_2 - (1 - t)^2 w_3, \\
w_3 \longmapsto (1 - t) w_1 + t w_3, \\
\end{array}
\right.
\end{equation}
\begin{equation} \label{a23}
a_{23} :
\left\{
\begin{array}{l}
w_1 \longmapsto w_1, \\
w_2 \longmapsto (1-t+t^2) w_2 + t (1 - t) w_3, \\
w_3 \longmapsto (1 - t) w_2 + t w_3, \\
\end{array}
\right.
\end{equation}
\begin{equation} \label{a23-1}
a_{23}^{-1} :
\left\{
\begin{array}{l}
w_1 \longmapsto  w_1,\\
w_2 \longmapsto t^{-1} w_2 +  (1 - t^{-1}) w_3, \\
w_3 \longmapsto t^{-1} (1 - t^{-1}) w_2 + (1 - t^{-1} + t^{-2}) w_3. \\
\end{array}
\right.
\end{equation}
Let us denote by $\widehat{\rho}_B$ the representation (if it exists)
$$
\widehat{\rho}_B : \widehat{B}_n \longrightarrow GL(W_n)
$$
such that 
\begin{equation}
\rho_B= \widehat{\rho}_B \circ\phi : B_n\to GL(W_n).
\label{rho_tilde}
\end{equation}
\begin{prop}
For $n=3$ the 
representation $\widehat{\rho}_B$ such that condition \ref{rho_tilde} holds  exists only if 
we consider the specialization of the Burau representation with $t=1$. In this case $\widehat{\rho}_B$
is trivial on $\widehat{P}_3$.
Hence, the image $\widehat{\rho}_B(B_3)$ is isomorphic to the symmetric group $S_3$.
\end{prop}
\begin{proof}
To obtain a representation  $\widehat{\rho}_B(B_3)$ we must have
the following relations among the automorphisms $a_{i,j}$ (\ref{a12})-(\ref{a23}) of $W_3$:
$$
[a_{13}, a_{13}^{a_{23}}] = 1,~~~[a_{23}, a_{23}^{a_{13}}] = 1,
$$
which are equivalent to the following relations
$$
a_{13} a_{13}^{a_{23}} = a_{13}^{a_{23}}  a_{13},~~~a_{23} a_{23}^{a_{13}} = a_{23}^{a_{13}}  a_{23}.
$$
From the expressions of the automorphisms (\ref{a12})-(\ref{a23-1}) we obtain
$$
a_{23}^{-1} a_{13} a_{23} :
\left\{
\begin{array}{l}
w_1 \longmapsto  (1 - t + t^2) w_1 + t (1-t)^2 w_2 + t^2 (1-t) w_3,\\
w_2 \longmapsto  w_2, \\
w_3 \longmapsto t^{-1} (1 - t) w_1 - t^{-1} (1 - t)^{2} w_2 + t w_3, \\
\end{array}
\right.
$$
and
$$
a_{13} a_{13}^{a_{23}} :
\left\{
\begin{array}{ll}
w_1 \longmapsto & (2 - 4t + 4t^2 - 2t^3 + t^4) w_1 + (1-t)^2 (-1 - t^2 + t^3) w_2 +\\
 & + t^2 (1-t) (2-t+t^2) w_3,\\
w_2 \longmapsto & (1 - t)^2 (- t^{-1} + 2 - t + t^2) w_1 + [(1-t)^4 (t + t^{-1}) + 1] w_2 +\\
 & + t (1-t)^2 (-1 + t - t^2) w_3, \\
w_3 \longmapsto & (1 - t) (2 - t + t^2) w_1 + (1 - t)^2 [-1 + t - t^{2}] w_2 +\\
&  + t^2 (2 - 2 t + t^2) w_3, \\
\end{array}
\right.
$$
$$
a_{13}^{a_{23}} a_{13} :
\left\{
\begin{array}{ll}
w_1 \longmapsto & (1 - t + 2 t^3 - 2t^4 + t^5) w_1 + t(1-t)^2 w_2 +\\
& + t (1-t) (1 - 2t + 5t^2 - 3 t^3 + t^4) w_3,\\
w_2 \longmapsto & (1 - t)^2 w_1 +  w_2 - (1-t)^2 w_3, \\
w_3 \longmapsto  & (1 - t) (2 - t + t^2) w_1 - t^{-1} (1 - t)^2  w_2 + \\
& + [(1-t)^2 (1 + t - 2 t^2 + t^3) + t^2] w_3. \\
\end{array}
\right.
$$
In order to satisfy relation $a_{13} a_{13}^{a_{23}} = a_{13}^{a_{23}} a_{13}$
the following system of equations should have a solution
$$
\left\{
\begin{array}{l}
1 - 3 t + 4t^2 - 4t^3 + 3t^4 - t^5 = 0,\\
(1 - t)^2 (-1 - t - t^2 + t^3) = 0, \\
t (1-t)^5 = 0, \\
(1 - t)^2 (-t^{-1} + 1 - t + t^2) = 0, \\
t^{-1} (1 - t)^4 (1 + t^2) = 0, \\
(1 - t)^2 (1 - t + t^2 - t^3) = 0, \\
(1 - t)^2 (-1 + t - t^2 + t^{-1}) = 0, \\
1 - t - 4t^2 + 8 t^3 - 5 t^4 + t^5 = 0. \\
\end{array}
\right.
$$
This system has a solution only if $t = 1$. In this case, automorphisms $a_{12}$, $a_{13}$, $a_{23}$ are equal to the identity automorphism.
\end{proof}

\subsection{Linear representation of \texorpdfstring{$K_n$}{Lg}}
 We know that $K_2$ is the
free 2-step nilpotent group of rank 2. Hence, for every non-zero integers $a$
and $b$ the map of $K_2$ defined on the generators $x_1$, $x_2$ by the formulas
$$
x_1 \longmapsto A =
\left(
  \begin{array}{ccc}
    1 & a & 0 \\
    0 & 1 & 0 \\
        0 & 0 & 1 \\
  \end{array}
\right),~~~
x_2 \longmapsto B =
\left(
  \begin{array}{ccc}
    1 & 0 & 0 \\
    0 & 1 & b \\
        0 & 0 & 1 \\
  \end{array}
\right)
$$
defines a faithful representation to the unitriangular group
 $$K_2 \longrightarrow UT_3(\mathbb{Z}).$$
Note that in this case
$$
[x_1, x_2] \longmapsto [A, B] =
\left(
  \begin{array}{ccc}
    1 & 0 & a b \\
    0 & 1 & 0 \\
        0 & 0 & 1 \\
  \end{array}
\right).
$$

\begin{question}
Is there exist a faithful representation of $K_n$, $n > 2$ into $UT_n(\mathbb{Z})$?
\end{question}
\begin{lem}
 The following basis commutators of weight 3 are trivial in $K_3$
$$
[[x_2, x_1], x_1],~~[[x_2, x_1], x_2],~~[[x_3, x_1], x_1],~~[[x_3, x_1], x_3],~~[[x_3, x_2], x_2],~~[[x_3, x_2], x_3],
$$
and the following basis commutators are non-trivial
$$
[[x_2, x_1], x_3],~~[[x_3, x_1], x_2]. 
$$
\end{lem}

\begin{proof}
Let us prove, for example, that $[[x_2, x_1], x_2] = 1$. Indeed, using commutator identities we have
$$
[[x_2, x_1], x_2] = [x_2^{-1} x_1^{-1} x_2 x_1, x_2] = [x_2^{-1} x_2^{x_1}, x_2] = [x_2^{-1}, x_2] [x_2^{x_1}, x_2]^{x_2^{-1}} = 1.
$$
To prove that the commutator $[[x_2, x_1], x_3]$ is non-trivial, we define an embedding of  $K_3$ into  $\widehat P_4$ (see \cite{HL}) by the rules
$$
x_1 \to a_{14},~~x_2 \to a_{24},~~x_3 \to a_{34}.
$$
Then  using  the fact that
the homomorphism
 $$
 \hat \nu_n:\widehat B_n \to\operatorname{Aut} K_n,
 $$
 is a monomorphism \cite{Ming},
we find the automorphism $\hat \nu_4 ([[a_{24}, a_{14}], a_{34}])$ and check that it is non-identity.
 The proofs of the other statements are similar.
\end{proof}

\section{Torsion  in \texorpdfstring{$\widehat{B}_n$}{Lg}}

V.Ya.~Lin formulated the following question in the Kourovka Notebook \cite{Kourovka}.
\begin{question}(V.Lin, Question 14.102 c))
Is there a non-trivial epimorphism of $B_n$ onto a  non-abelian
group without torsion?
\end{question}

The answer to this question  can be found in
\cite{L}

We conjecture that the group $\widehat{B}_n$, $n \geq 3$, does not
have torsion and since there exists the epimorphism
$B_n \longrightarrow \widehat{B}_n$,  the group $\widehat{B}_n$ will be
another example that answers Lin's question.
 In this section we  prove that $\widehat{B}_3$
 does not have torsion.

Let $\widehat{P}_3$, $\widehat{U}_2$, $\widehat{U}_3$ be the
images of $P_3$, $U_2$, $U_3$
by the canonical  epimorphism $B_3 \longrightarrow \widehat{B}_3$.
Denote by $b_{ij}$, $1 \leq i < j \leq 3$ the images of $a_{ij}$,
$1 \leq i < j \leq 3$, by this epimorphism. Then
$\widehat{U}_2 = \langle b_{12} \rangle$ is the
infinite cyclic group and
$$
\widehat{U}_3 = \langle b_{13}, b_{23}~||~ [b_{13}, b_{13}^{b_{23}}] =
[b_{23}, b_{23}^{b_{13}}] = 1 \rangle =
$$
$$
= \langle b_{13}, b_{23}~||~ [b_{13}, b_{13} [b_{13}, b_{23}]] =
[b_{23}, b_{23} [b_{23}, b_{13}]] = 1 \rangle.
$$
Using commutator identities or direct calculations we see that the last two relations are equivalent to
the following relation
$$
[[b_{23}, b_{13}], b_{23}] = [[b_{23}, b_{13}], b_{13}] = 1.
$$
Hence, $\widehat{U}_3$ is a free 2-step nilpotent group of rank 2 and
so, every element $g \in \widehat{U}_3$  has a unique presentation of
 the form
$$
g = b_{13}^{\alpha} b_{23}^{\beta} [b_{23}, b_{13}]^{\gamma}
$$
for some integers $\alpha, \beta, \gamma$. The same way as in the case of
classical braid group, $\widehat{U}_3$ is a normal subgroup of
$\widehat{P}_3$ and the action of $\widehat{U}_2$ is defined in the
following lemma.
\begin{lem}
The action of $\widehat{U}_2$ on $\widehat{U}_3$ is given by the formulas
$$
b_{13}^{b_{12}^k} = b_{13} [b_{23}, b_{13}]^k,~~b_{23}^{b_{12}^k} =
b_{23} [b_{23}, b_{13}]^{-k},~~[b_{23}, b_{13}]^{b_{12}^k} =
[b_{23}, b_{13}],~~k \in \mathbb{Z}. \ \square
$$
\end{lem}

The action of the generators $\sigma_1$ and $\sigma_2$ of $\widehat{B}_3$ on
$\widehat{P}_3$ is given in the next lemma.
\begin{lem}
 The following conjugation formulas hold in  $\widehat{B}_3$
$$
b_{12}^{\sigma_1^{\pm 1}} = b_{12}, ~~ b_{13}^{\sigma_1} = b_{23}
[b_{23}, b_{13}]^{-1},~~b_{23}^{\sigma_1} = b_{13},
b_{13}^{\sigma_1^{-1}} = b_{23},~~b_{23}^{\sigma_1^{-1}} = b_{13}
[b_{23}, b_{13}]^{-1},
$$
$$
[b_{23}, b_{13}]^{\sigma_1^{-1}} = [b_{23}, b_{13}]^{-1},
$$
$$
b_{12}^{\sigma_2} = b_{13} [b_{23}, b_{13}]^{-1},~~b_{13}^{\sigma_2} =
b_{12}, ~~ b_{23}^{\sigma_2^{\pm 1}} = b_{23},~~
b_{12}^{\sigma_2^{-1}} = b_{13},
~~b_{13}^{\sigma_2^{-1}} = b_{12} [b_{23}, b_{13}]^{-1},
$$
$$
[b_{23}, b_{13}]^{\sigma_2^{-1}} = [b_{23}, b_{13}]^{-1}. \quad \square
$$
\end{lem}
Let us denote by $\Lambda_3 = \{ e, \sigma_1, \sigma_2, \sigma_2 \sigma_1,
\sigma_1 \sigma_2, \sigma_1 \sigma_2 \sigma_1 \}$ the set of representatives of
$\widehat{P}_3$ in $\widehat{B}_3$.
Then every element in $\widehat{B}_3$ can be written  in the form
$$
b_{12}^{\alpha} b_{13}^{\beta} b_{23}^{\gamma} z^{\delta}
\lambda,~\mbox{where}~\alpha, \beta, \gamma, \delta \in
\mathbb{Z},~~z = [b_{23}, b_{13}], ~~\lambda \in \Lambda_3.
$$
\begin{thm}
The group $\widehat{B}_3$ is torsion-free.
\end{thm}
\begin{proof}
The group $\widehat{P}_3$ does not have torsion. Hence, if $\widehat{B}_3$
has elements of finite order, then they have the form
$$
b_{12}^{\alpha} b_{13}^{\beta} b_{23}^{\gamma} z^{\delta} \lambda,~~
\lambda \in \Lambda_3 \setminus \{ e \}.
$$
Every element which is conjugate with an element of finite order has a finite
order. Taking into account the following formulas
$$
\sigma_1^{-1} \cdot \sigma_2 \cdot \sigma_1 = b_{12}^{-1}
\sigma_1 \sigma_2 \sigma_1,
~~\sigma_2 \sigma_1 \cdot \sigma_2 \cdot \sigma_1^{-1} \sigma_2^{-1} =
\sigma_1,~~
\sigma_1^{-1} \cdot \sigma_1 \sigma_2 \cdot \sigma_1 = \sigma_2 \sigma_1,
$$
it is sufficient to consider only two cases: $\lambda = \sigma_2$ and
$\lambda = \sigma_1 \sigma_2$.
Let $\lambda = \sigma_2$, take $g = b_{12}^{\alpha} b_{13}^{\beta}
b_{23}^{\gamma} z^{\delta} \sigma_2$. Then we have
$$
g^2 = b_{12}^{\alpha + \beta} b_{13}^{\alpha + \beta}  b_{23}^{2\gamma + 1}
z^{ \alpha \gamma + \beta (\beta - \gamma + \alpha -1)}.
$$
If $g^2 = 1$, then $\alpha + \beta = 0$ and we have
$$
g^2 = b_{23}^{2\gamma + 1} z^{2 \alpha \gamma + \alpha}.
$$
Since  $2\gamma + 1$ cannot be zero for integer $\gamma$, the elements
of this form cannot be of finite order.

Let $\lambda = \sigma_1 \sigma_2$. Then we have
$$
(\sigma_1 \sigma_2)^2 = b_{12} \sigma_2 \sigma_1,~~(\sigma_1 \sigma_2)^3 = b_{12} b_{13} b_{23}.
$$
We calculate
\begin{multline*}
g^3 = (b_{12}^{\alpha} b_{13}^{\beta} b_{23}^{\gamma} z^{\delta} \sigma_1 \sigma_2)^3 =\\
b_{12}^{\alpha + \beta + \gamma + 1} b_{13}^{\alpha + \beta + \gamma + 1}
b_{23}^{\alpha + \beta + \gamma + 1} z^{\alpha(\alpha+2\gamma-\beta)+ \beta^2 +
\gamma^2 - \beta \gamma + 3\delta+3\beta}.
\end{multline*}
If $g^3 = 1$, then the following system of linear equations has a solution over $\mathbb{Z}$
$$
\left\{
\begin{array}{l}
\alpha + \beta + \gamma + 1 = 0,\\
\alpha(\alpha+2\gamma-\beta)+ \beta^2 + \gamma^2 - \beta \gamma + 3\delta+3\beta = 0. \\
\end{array}
\right.
$$
From the first equation one gets: $\alpha = -1 - \beta - \gamma$.
Inserting this expression of  $\alpha$  into the second equation, we have
$$
3 (\beta^2 +  2 \beta + \delta) + 1 = 0.
$$
However, this equation does not have integer solutions.
\end{proof}

\section{Open problems}

The homotopy virtual braid group $\widehat{VB}_n$ was defined in \cite{D}.

\begin{question}
1) Is it possible to construct a normal form for words, representing elements of $\widehat{VB}_n$?

2) Is there a faithful representation $\widehat{VB}_n \longrightarrow Aut(G_n)$ for some group $G_n$?

3) Is it possible to define Milnor invariants for homotopy virtual links?

\end{question}

\section{Acknowledgements}
The first author is supported by the Laboratory of Topology and Dynamics of Novosibirsk State University (contract no. 14.Y26.31.0025 with the Ministry of Education and Science of the Russian Federation).
The last author is partially supported by High-level Scientific Research Foundation of
Hebei Province and a grant (No. 11971144) of NSFC of China.

The second author thanks Emmanuel Graff for helpful discussions.

\end{document}